\documentclass[11pt,reqno]{amsart}
\usepackage[letterpaper,margin=1in,footskip=0.25in]{geometry}
\usepackage{mathrsfs}
\usepackage{microtype}
\usepackage{mathtools}
\usepackage{enumitem}
\PassOptionsToPackage{dvipsnames}{xcolor}
\usepackage{tikz-cd}
\PassOptionsToPackage{pdfusetitle,pagebackref,colorlinks}{hyperref}
\usepackage{bookmark}
\hypersetup{citecolor=OliveGreen,linkcolor=Mahogany,urlcolor=Plum}
\usepackage[hyphenbreaks]{breakurl}

\newtheorem{theorem}{Theorem}[section]
\newtheorem{lemma}[theorem]{Lemma}
\newtheorem{proposition}[theorem]{Proposition}
\newtheorem{corollary}[theorem]{Corollary}

\newtheorem{innercustomthm}{Theorem}
\newenvironment{customthm}[1]
  {\renewcommand\theinnercustomthm{#1}\innercustomthm}
  {\endinnercustomthm}

\theoremstyle{definition}
\newtheorem{definition}[theorem]{Definition}
\newtheorem{example}[theorem]{Example}
\newtheorem*{question*}{Question}

\theoremstyle{remark}
\newtheorem{remark}[theorem]{Remark}

\newtheoremstyle{cited}{.5\baselineskip\@plus.2\baselineskip\@minus.2\baselineskip}{.5\baselineskip\@plus.2\baselineskip\@minus.2\baselineskip}{\itshape}{}{\bfseries}{\bfseries .}{5pt plus 1pt minus 1pt}{\thmname{#1}\thmnumber{ #2}\thmnote{ #3}}
\theoremstyle{cited}
\newtheorem{citedthm}[theorem]{Theorem}
\newtheorem{citedconj}[theorem]{Conjecture}
\newtheorem*{customthm*}{Theorem}

\DeclareMathOperator{\mult}{mult}
\DeclareMathOperator{\Char}{char}
\DeclareMathOperator{\Sym}{Sym}
\newcommand{\PP}{\mathbf{P}}
\newcommand{\QQ}{\mathbf{Q}}
\newcommand{\RR}{\mathbf{R}}
\newcommand{\cO}{\mathcal{O}}
\newcommand{\sI}{\mathscr{I}}
\newcommand{\sP}{\mathscr{P}}
\newcommand{\fa}{\mathfrak{a}}
\newcommand{\fm}{\mathfrak{m}}
\newcommand{\longtwoheadrightarrow}{\mathrel{\text{\tikz \draw [-cm double to] (0,0) (0.05em,0.5ex) -- (1.525em,0.5ex);}\hspace{0.05em}}}

\begin{document}
\title[Frobenius--Seshadri constants and characterizations of projective
space]{Frobenius--Seshadri constants and\\characterizations of projective space}
\author{Takumi Murayama}
\thanks{This material is based upon work supported by the National Science
Foundation under Grant No.\ DMS-1265256}
\address{Department of Mathematics\\University of Michigan\\
Ann Arbor, MI 48109-1043, USA}
\email{\href{mailto:takumim@umich.edu}{takumim@umich.edu}}
\urladdr{\url{http://www-personal.umich.edu/~takumim/}}

\makeatletter
  \hypersetup{pdfsubject=\@subjclass,pdfkeywords=\@keywords}
\makeatother

\begin{abstract}
  We introduce higher-order variants of the Frobenius--Seshadri constant due to
  Musta\c{t}\u{a} and Schwede, which are defined for ample line bundles in
  positive characteristic. These constants are used to show that Demailly's
  criterion for separation of higher-order jets by adjoint bundles also holds
  in positive characteristic. As an application, we give a characterization of
  projective space using Seshadri constants in positive characteristic, which
  was proved in characteristic zero by Bauer and Szemberg. We also discuss
  connections with other characterizations of projective space.
\end{abstract}

\maketitle

\section{Introduction}\label{s:intro}
Let $L$ be an ample line bundle on a smooth projective variety $X$ defined over
an algebraically closed field $k$. Demailly in \cite[\S6]{Dem92}
introduced the \textsl{Seshadri constant} $\varepsilon(L;x)$, which measures the
local positivity of
$L$ at a closed point $x \in X$. If $\mu\colon X' \to X$ is the
blow-up of $X$ at $x$ with exceptional divisor $E$, then the Seshadri constant
is
\[
  \varepsilon(L;x) \coloneqq \sup\bigl\{t \in \RR_{\ge0} \bigm\vert \mu^*(L)
  (-tE)\ \text{is nef} \bigr\}.
\]
Seshadri constants have received much attention since their inception: see
\cite{BDHKKSS09} and \cite[Ch.\ 5]{Laz04}.
\par Part of this interest in Seshadri constants stems from the fact that they
give effective positivity statements for adjoint bundles. We will
be particularly interested in how Seshadri constants can determine
when adjoint bundles
separate higher-order jets. Recall that we say $L$ \textsl{separates $\ell$-jets}
if the restriction map
\begin{equation}\label{eq:defsepjets}
  H^0(X,L) \longrightarrow H^0(X,L \otimes \cO_X/\fm_x^{\ell+1})
\end{equation}
is surjective, where $\fm_x \subset \cO_X$ is the ideal defining $x$.
Algebraically, this says $\ell$th-order Taylor polynomials
at $x$ are restrictions of global sections, and geometrically, this
says $L$ separates $\ell$th-order tangent directions at $x$.
\par We can now state our first main result. Let $\omega_X$ denote the
canonical bundle on $X$.
\begin{customthm*}[\ref{thm:sepljetsdem}]
  Let $L$ be an ample line bundle on a smooth projective variety $X$ of
  dimension $n$ defined
  over an algebraically closed field of positive characteristic. If
  $\varepsilon(L;x) > n + \ell$ at a closed point $x \in X$, then $\omega_X
  \otimes L$ separates $\ell$-jets at $x$.
\end{customthm*}
Demailly showed this result in characteristic zero using the Kawamata--Viehweg
vanishing theorem \cite[Prop.\ 6.8$(a)$]{Dem92}. Our
contribution is that the same result holds in positive characteristic.
\par As an application of this result, we prove the following:
\begin{customthm*}[\ref{thm:charpn}]
  Let $X$ be a smooth Fano variety of dimension $n$ defined over an
  algebraically closed field of positive characteristic. If there exists a
  closed point $x \in X$ with $\varepsilon(\omega_X^{-1};x) \ge n+1$, then $X$
  is isomorphic to the $n$-dimensional projective space~$\PP^n$.
\end{customthm*}
\par Bauer and Szemberg showed the analogous statement in characteristic zero
as an application of the proof of Demailly's result \cite[Thm.\ 2]{BS09}. One
interesting feature of this theorem is that it only
requires a positivity condition on the anti-canonical bundle $\omega_X^{-1}$ at
\emph{one} point $x \in X$. In characteristic zero, Liu and Zhuang
in \cite[Thm.~2]{LZ16} generalized \cite[Thm.\ 2]{BS09} to $\QQ$-Fano
varieties; see Remark \ref{rem:lz16} for a comparison between their result and
Theorem \ref{thm:charpn}.
\par There is an interesting connection between Theorem \ref{thm:charpn}
and the Mori--Mukai conjecture, which states that if
$X$ is a smooth Fano variety of dimension $n$ such that the canonical divisor $K_X$
satisfies $(-K_X \cdot C) \ge n+1$ for all rational curves $C \subseteq X$,
then $X$ is isomorphic to $\PP^n$. In characteristic zero, Cho, Miyaoka, and
Shepherd-Barron's proof of the conjecture \cite[Cor.\ 0.4]{CMSB02} implies
\cite[Thm.\ 2]{BS09}; see
Proposition \ref{prop:charpnallpts}$(\ref{prop:charpnallptschar0})$. In positive
characteristic, the conjecture is still open, so instead one must assume
the lower bound in Theorem \ref{thm:charpn} holds at \emph{all} closed
points in order to use a weaker result due to Kachi and Koll\'ar \cite[Cor.\
3]{KK00}; see
Proposition \ref{prop:charpnallpts}$(\ref{prop:charpnallptscharp})$.
\par On the other hand, this connection raises the question of whether the
opposite relationship holds. More precisely, we ask the following:
\begin{question*}
  Let $X$ be a smooth Fano variety of dimension $n$ defined over an
  algebraically closed field. If
  $(-K_X \cdot C) \ge n+1$
  for every rational curve $C \subseteq X$, then is there a closed
  point $x \in X$ with
  $(-K_X \cdot C) \ge (\mult_x C) \cdot (n+1)$
  for all reduced and irreducible curves $C \subseteq X$ passing through $x$?
\end{question*}
\par This latter condition is equivalent to having
$\varepsilon(\omega_X^{-1};x) \ge n+1$ by \cite[Prop.\ 5.1.5]{Laz04}. In
characteristic zero, \cite[Cor.\ 0.4]{CMSB02} answers this question
affirmatively. If one could do this independently of their result, then
\cite[Thm.\ 2]{BS09} would give an alternative proof of the Mori--Mukai
conjecture in characteristic zero, and Theorem \ref{thm:charpn} would
resolve the conjecture in positive characteristic. See
\S\ref{s:charpncomparison} for further discussion.
\par The proofs of Theorems \ref{thm:sepljetsdem} and \ref{thm:charpn}
use a new variant of the
Seshadri constant. To motivate our definition, we recall the
following alternative characterization of the Seshadri constant in terms of
separation of jets \cite[Prop.\ 5.1.17]{Laz04}: if for each integer
$m$, we denote by $s(L^m;x)$ the largest integer $\ell$ such that $L^m$
separates $\ell$-jets, then we have the equalities
\[
  \varepsilon(L;x) = \sup_{m \ge 1} \frac{s(L^m;x)}{m} = \lim_{m \to \infty}
  \frac{s(L^m;x)}{m}.
\]
In positive characteristic, Musta\c{t}\u{a} and Schwede \cite{MS14} defined
the Frobe\-nius--Seshadri constant essentially by replacing
ordinary powers of $\fm_x$ in the definition of separation of jets
\eqref{eq:defsepjets} with Frobenius powers, in order to
take advantage of the Frobenius morphism. Using their definition, Musta\c{t}\u{a}
and Schwede were able to recover Theorem~\ref{thm:sepljetsdem} when
$\ell=0$, and deduce global generation and very ampleness results for adjoint
bundles \cite[Thm.\ 3.1]{MS14}. However, the full statement of
Theorem \ref{thm:sepljetsdem} remained out of reach; see
Remark \ref{rem:comparisonwithms}.
\par Our solution is to introduce higher-order variants of the
Frobenius--Seshadri constant, which mix both ordinary and Frobenius powers of
$\fm_x$. This allows us to more directly deduce separation of higher-order jets.
For each integer $\ell \ge 0$ and $m \ge 1$, let $s_F^\ell(L^m;x)$ be the
largest integer $e \ge 0$ such that the restriction map
\[
  H^0(X,L^m) \longrightarrow H^0\bigl(X,L^m \otimes
  \cO_X/(\fm_x^{\ell+1})^{[p^e]}\bigr)
\]
is surjective, where $(\fm_x^{\ell+1})^{[p^e]}$ denotes the $e$th Frobenius
power of $\fm_x^{\ell+1}$. Then, the
\textsl{$\ell$th Frobenius--Seshadri constant} of $L$ at $x$ is
\[
  \varepsilon_F^\ell(L;x) \coloneqq \sup_{m \ge 1}
  \frac{p^{s_F^\ell(L^m;x)} - 1}{m/(\ell+1)} = \limsup_{m \to \infty}
  \frac{p^{s_F^\ell(L^m;x)} - 1}{m/(\ell+1)}.
\]
These constants are related to the ordinary Seshadri constant in the following
manner (Proposition \ref{prop:ms212}):
\begin{equation}\label{eq:compineq}
  \frac{\ell+1}{\ell+n} \cdot \varepsilon(L;x) \le \varepsilon_F^\ell(L;x) \le
  \varepsilon(L;x).
\end{equation}
Note that the $\ell$th Frobenius--Seshadri constant $\varepsilon_F^\ell(L;x)$
converges to the ordinary Seshadri constant $\varepsilon(L;x)$ as $\ell \to
\infty$.
\par Using the $\ell$th Frobenius--Seshadri constant, we
prove the following statement en route to proving Theorem\ \ref{thm:sepljetsdem}:
\begin{customthm*}[\ref{thm:sepljets}]
  Let $L$ be an ample line bundle on a smooth projective variety $X$ defined
  over an algebraically closed field of positive characteristic.
  If $\varepsilon_F^\ell(L;x) > \ell + 1$ at a closed point $x \in X$,
  then $\omega_X \otimes L$ separates $\ell$-jets at $x$.
\end{customthm*}
\noindent By the comparison \eqref{eq:compineq}, Theorem
\ref{thm:sepljets} immediately implies Theorem \ref{thm:sepljetsdem}. Our proof
of Theorem \ref{thm:sepljets} also works for singular varieties that are
\textsl{$F$-injective}; see Remark \ref{rem:singvar}. $F$-injective varieties
are related to varieties with Du Bois singularities in characteristic zero
\cite{Sch09a}.
\par Given Theorem \ref{thm:sepljets}, it would be very interesting
to have non-trivial lower bounds for any of the aforementioned versions of the
Seshadri constant at very general points of $X$. In
characteristic zero, it is conjectured that
$\varepsilon(L;x) \ge 1$ at all very general $x \in X$. This is known if $n
= \dim X =
2$ \cite{EL93}; if $n \ge 3$, then only the lower bound $1/n$ is known
\cite[Thm.\ 1]{EKL95}. However, the proofs of both results rely heavily
on the characteristic zero assumption, and in arbitrary characteristic, we are
only aware of the lower bound $\varepsilon(L;x) \ge 2/\bigl(1 +
\sqrt{4\sigma(L)+13}\bigr)$ for arbitrary points on surfaces, where
\[
  \sigma(L) \coloneqq \inf\bigl\{s \in \RR \bigm\vert sL - K_X\ \text{is
  nef}\bigr\}
\]
is the \textsl{canonical slope} of $L$ \cite[Thm.\ 3.1]{Bau99}.
\subsection*{Outline}
Our paper is structured as follows: In \S\ref{s:defs}, we define the $\ell$th
Frobenius--Seshadri constant, and prove its basic properties. Most of what
we prove is modeled after \cite[\S2]{MS14}, which studies what would be the
zeroth Frobenius--Seshadri constant in our notation. In \S\ref{s:adjoint}, we
prove Theorem \ref{thm:sepljets}. The main technical tool is the trace map
$T\colon F_*\omega_X \to \omega_X$ associated to the (absolute) Frobenius
morphism. Finally, in \S\ref{s:charpn}, we prove Theorem \ref{thm:charpn},
following \cite{BS09}.
\subsection*{Notation}
A \textsl{variety} is a reduced and irreducible separated scheme of finite type
defined over an algebraically closed field $k$ of characteristic $p > 0$, unless
stated otherwise. We denote by $X$ a positive-dimensional projective variety,
and denote by $F\colon X \to X$ the
\textsl{(absolute) Frobenius morphism,} which is given by the identity map on
points, and the $p$-power map
\[
  \begin{tikzcd}[row sep=0,column sep=1.6em]
    \cO_X(U) \rar & F_*\cO_X(U)\\
    f \rar[mapsto] & f^p
  \end{tikzcd}
\]
on structure sheaves, where $U \subseteq X$ is an open set.
If $\fa \subseteq \cO_X$ is a coherent ideal sheaf, we define the \textsl{$e$th
Frobenius power} $\fa^{[p^e]}$ to be the inverse image of
$\fa$ via the $e$th iterate of the Frobenius morphism.
Locally, if $\fa$ is generated by $(h_i)_{i \in I}$, then $\fa^{[p^e]}$ is
generated by $(h_i^{p^e})_{i \in I}$.
If $X$ is smooth, we denote by $\omega_X$ the canonical bundle on $X$ and
$K_X$ the canonical divisor on $X$.
\section{Definitions and preliminaries}\label{s:defs}
We start by recalling the definition of the (ordinary) Seshadri constant of a
line bundle $L$ at a point.
We adopt the ``separation of jets'' description of the Seshadri constant
as our definition, which is equivalent to the other definitions when the line
bundle $L$ is ample and the closed point $x$ is smooth
\cite[Prop.\ 5.1.17]{Laz04}.
\begin{definition}
  Let $L$ be a line bundle on a projective variety $X$, and let $x \in X$ be a
  closed point with defining ideal $\fm_x \subset \cO_X$. For all
  integers $\ell \ge 0$ and $m \ge 1$, we say that $L^m$ \textsl{separates
  $\ell$-jets at $x$} if the restriction map
  \[
    \rho^{\ell,0}_{L^m}\colon H^0(X,L^m) \longrightarrow H^0(X,L^m \otimes
    \cO_X/\fm_x^{\ell+1})
  \]
  is surjective. Let $s(L^m;x)$ be the largest integer $\ell \ge 0$ such
  that $L^m$ separates $\ell$-jets at $x$; if no such $\ell$
  exists, set $s(L^m;x) = -\infty$. The
  \textsl{Seshadri constant} of $L$ at $x$ is
  \begin{equation}\label{eq:defordseshadri}
    \varepsilon(L;x) \coloneqq \sup_{m \ge 1} \frac{s(L^m;x)}{m}.
  \end{equation}
\end{definition}
\par We now give our main definition, which is modeled after the above
interpretation
of the Seshadri constant in terms of separation of jets. This
definition combines both ordinary and Frobenius powers of the ideal $\fm_x$.
Compared to the Frobenius--Seshadri constant defined in \cite[Def.\ 2.4]{MS14},
our definition has the advantage of directly encoding information about
higher-order jets; see Remark \ref{rem:comparisonwithms}.
\begin{definition}
  Let $L$ be a line bundle on a projective variety $X$, and let $x \in X$ be a
  closed point with defining ideal $\fm_x \subset \cO_X$. For all
  integers $\ell,e \ge 0$ and $m \ge 1$, we say that $L^m$ \textsl{separates
  $p^e$-Frobenius $\ell$-jets at $x$} if the restriction map
  \[
    \rho^{\ell,e}_{L^m}\colon H^0(X,L^m) \longrightarrow H^0\bigl(X,L^m \otimes
    \cO_X/(\fm_x^{\ell+1})^{[p^e]}\bigr)
  \]
  is surjective. Let $s^\ell_F(L^m;x)$ be the largest integer $e \ge 0$ such
  that $L^m$ separates $p^e$-Frobenius $\ell$-jets at $x$; if no such $e$
  exists, set $s^\ell_F(L^m;x) = -\infty$. The
  \textsl{$\ell$th Frobenius--Seshadri constant} of $L$ at $x$ is
  \begin{equation}\label{eq:lthfrobsesh}
    \varepsilon^\ell_F(L;x) \coloneqq \sup_{m \ge 1} \frac{p^{s^\ell_F(L^m;x)} -
    1}{m/(\ell+1)}.
  \end{equation}
  We refer to the constants $\varepsilon^\ell_F(L;x)$ as
  \textsl{Frobenius--Seshadri constants}. Note that the zeroth
  Frobenius--Seshadri constant $\varepsilon^0_F(L;x)$ is
  the Frobenius--Seshadri constant defined in \cite[Def.\ 2.4]{MS14}.
\end{definition}
\par In the rest of this section, we will prove basic formal properties about
Frobenius--Seshadri constants, following \cite[\S2]{MS14}. The only
statements used explicitly in later sections are Lemmas \ref{lem:tensorpower}
and \ref{lem:tensorwithgg}, and Propositions \ref{prop:ms26}$(\ref{prop:ms26i})$
and \ref{prop:ms212}.
\subsection{Separation of Frobenius jets under tensor powers}
For the ordinary Seshadri constant, the supremum in \eqref{eq:defordseshadri}
is actually a limit when $L$ is ample \cite[p.\ 97]{Dem92}. This property follows
from Fekete's lemma \cite[Pt.\ I, n\textsuperscript{o} 98]{PS98}, since for all
positive integers $m$ and $n$, the sequence $s(L^m;x)$ satisfies the
superadditivity property (see, e.g., \cite[Lem.\ 3.7]{Ito13} for a proof)
\begin{equation}\label{eq:superadditive}
  s(L^{m+n};x) \ge s(L^m;x) + s(L^n;x).
\end{equation}
For Frobenius--Seshadri constants, we cannot have an analogous property, since
the supremum in \eqref{eq:lthfrobsesh} may not be a limit; see
Example \ref{ex:projectivespace}.
Our first goal is to find a replacement for this superadditivity property. This
will allow us to show that the supremum in \eqref{eq:lthfrobsesh} is actually a
limit supremum.
\par We start with the following observation about ideals in a ring
of characteristic $p > 0$, which will also be useful later. Note that the first
inclusion in \eqref{eq:calemincl} is a slight improvement on
\cite[Lem.\ 4.6]{Sch09}.
\begin{lemma}\label{lem:monomials}
  Let $R$ be a commutative ring of characteristic $p > 0$. Then, for any ideal
  $\fa$ generated by $n$ elements and for any non-negative integers $e$ and
  $\ell$, we have the sequence of inclusions
  \begin{equation}
    \fa^{\ell p^e+n(p^e-1)+1} \subseteq (\fa^{\ell+1})^{[p^e]} \subseteq
    \fa^{(\ell+1)p^e}.\label{eq:calemincl}
  \end{equation}
  Moreover, if $R$ is a regular local ring of dimension $n$, and
  $\fa$ is the maximal ideal of $R$, then
  \[
    \fa^{\ell p^e+n(p^e-1)} \not\subseteq (\fa^{\ell+1})^{[p^e]}.
  \]
\end{lemma}
\begin{proof}
  The second inclusion in \eqref{eq:calemincl} is clear; we want to
  show the first inclusion. Let $y_1,y_2,\ldots,y_n$ be a set of generators for
  $\fa$.  The ideal $\fa^{\ell p^e+n(p^e-1)+1}$ is generated
  by all elements of the form
  \begin{alignat}{4}
    &\prod_{i=1}^{n} y_i^{a_i} &&\quad& \text{such that} &\quad 
    \sum_{i=1}^n a_i &{}={}& \ell p^e+n(p^e-1)+1,
    \label{eq:monoordpower}
    \intertext{and the ideal $(\fa^{\ell+1})^{[p^e]}$ is generated by all
    elements of the form}
    &\prod_{i=1}^{n} y_i^{p^e b_i} &&\quad& \text{such that} &\quad 
    \sum_{i=1}^n b_i &{}={}& \ell+1.\label{eq:monomixedpower}
  \end{alignat}
  We want to show that the elements \eqref{eq:monoordpower} are divisible by
  some elements of the form \eqref{eq:monomixedpower}.
  By the division algorithm, we may write $a_i = a_{i,0} + p^e a_{i}'$ for some
  non-negative integers $a_{i,0}$ and $a_i'$
  such that $0 \le a_{i,0} \le p^{e}-1$. Then,
  \[
    \prod_{i=1}^{n} y_i^{a_i} = \prod_{i=1}^n y_i^{a_{i,0}} \cdot
    \prod_{i=1}^n y_i^{p^e a_{i}'},
  \]
  and since $a_{i,0} \le p^e - 1$, we have that $\sum_{i=1}^n a_{i,0} \le n(p^e -
  1)$. Thus, we have the inequality
  \[
    \ell p^e + n(p^e - 1) + 1 = \sum_{i=1}^n a_i \le n(p^e - 1) +
    \sum_{i=1}^n p^e a_i',
  \]
  which implies $\ell + p^{-e} \le \sum_{i=1}^n a_i'$. Since the right-hand side
  of this inequality is an integer, we have that $\ell + 1 \le
  \sum_{i=1}^n a_i'$, i.e., the element $\prod_{i=1}^n y_i^{p^ea_i'}$ is
  divisible by one of the form \eqref{eq:monomixedpower}. Thus, each element of
  the form in \eqref{eq:monoordpower} is divisible by one of the form in
  \eqref{eq:monomixedpower}.
  \par Now suppose $R$ is a regular local ring of dimension $n$, and $\fa$
  is the maximal ideal of $R$. Let $y_1,y_2,\ldots,y_n$ be a regular system of
  parameters. Then, we have
  \[
    y_{i_0}^{\ell p^e} \cdot \prod_{i=1}^n y_i^{p^e-1} \in
    \fa^{\ell p^e+n(p^e-1)}
  \]
  for any $i_0 \in \{1,2,\ldots,n\}$. This monomial
  does not lie in $(\fa^{\ell+1})^{[p^e]}$ since its image is not in the
  extension of $(\fa^{\ell+1})^{[p^e]}$ in the completion of $R$ at $\fa$,
  which is isomorphic to a formal power series ring with variables
  $y_1,y_2,\ldots,y_n$ by the Cohen structure theorem.
\end{proof}
\begin{lemma}[cf.\ {\cite[Lem.\ 2.5]{MS14}}]\label{lem:tensorpower}
  Let $L$ be a line bundle on a projective variety $X$, and let $x \in X$ be a
  closed point. Suppose $L^m$
  separates $p^e$-Frobenius $\ell$-jets at $x$ for some
  integers $m \ge 1$, $e \ge 1$, and $\ell \ge 0$, and denote
  \[
    d_r = \frac{p^{re}-1}{p^e-1}
  \]
  for each positive integer $r$. Then, $L^{md_r}$ separates
  $p^{re}$-Frobenius $\ell$-jets at $x$ for all $r \ge 1$.
\end{lemma}
\begin{proof}
  We want to show the restriction maps
  \[
    \varphi_r \colon H^0(X,L^{md_r}) \longrightarrow H^0\bigl(X,L^{md_r}
    \otimes \cO_X/(\fm_x^{\ell+1})^{[p^{re}]}\bigr)
  \]
  are surjective for all positive integers $r$. We prove this by induction on
  $r$. The case $r = 1$ is true by assumption, so we consider the case when $r
  \ge 2$.
  \par Let $\widetilde{y}_1,\widetilde{y}_2,\ldots,\widetilde{y}_n$ generate
  $\fm_x\cdot\cO_{X,x}$. After choosing an isomorphism $L^{m}_x \simeq
  \cO_{X,x}$, we can make the identification
  \[
    L^{md_r} \otimes \cO_X/(\fm_x^{\ell+1})^{[p^{re}]} \simeq
    \cO_X/(\fm_x^{\ell+1})^{[p^{re}]}.
  \]
  Then, $L^{md_r} \otimes \cO_X/(\fm_x^{\ell+1})^{[p^{re}]}$ is
  generated as a vector space over $k$ by the residue classes
  \begin{equation}\label{eq:highermonomials}
    y_1^{a_1} y_2^{a_2} \cdots y_n^{a_n} \in L^{md_r} \otimes
    \cO_X/(\fm_x^{\ell+1})^{[p^{re}]}
  \end{equation}
  of the monomials $\widetilde{y}_1^{a_1}\widetilde{y}_2^{a_2} \cdots
  \widetilde{y}_n^{a_n} \in L^{md_r}_x$, where the $a_i$ can be any
  non-negative integers.
  As in the proof of Lemma \ref{lem:monomials}, write $a_i = a_{i,0} + p^e a_i'$
  for each $i$, where $0 \le a_{i,0} \le p^e-1$, so that
  \[
    y_1^{a_1} y_2^{a_2} \cdots y_n^{a_n} = \prod_{i=1}^n y_i^{a_{i,0}} \cdot
    \prod_{i=1}^n y_i^{p^ea_{i}'}.
  \]
  We will show that the
  elements in \eqref{eq:highermonomials} lie in the image of $\varphi_r$
  by descending induction on $S \coloneqq \sum_{i=1}^n a_i'$. By
  Lemma \ref{lem:monomials}, if $S \ge (\ell+n) p^{(r-1)e}$, then
  $y_1^{a_1} y_2^{a_2} \cdots y_n^{a_n} \equiv 0 \bmod
  (\fm_x^{\ell+1})^{[p^{re}]}$, and so there is nothing to show. It therefore
  suffices to consider the inductive case, when $S \le (\ell+n) p^{(r-1)e} - 1$.
  \par By the assumption that $\varphi_1$ is surjective, we know that there
  exists $t_1 \in H^0(X,L^{m})$ such that its germ $t_{1,x} \in L_x^m$ satisfies
  \begin{equation}\label{eq:t1lift}
    \prod_{i=1}^n \widetilde{y}_i^{a_{i,0}} - t_{1,x} \in (\fm_x^{\ell+1})^{[p^e]} \otimes
    L^{m}_x.
  \end{equation}
  By the inductive hypothesis with respect to $r$, we know that
  $\varphi_{r-1}$ is surjective, so there exists $t_2 \in
  H^0(X,L^{md_{r-1}})$ such that its germ $t_{2,x} \in L_x^{md_{r-1}}$ satisfies
  \[
    \prod_{i=1}^n \widetilde{y}_i^{a_{i}'} - t_{2,x} \in (\fm_x^{\ell+1})^{[p^{(r-1)e}]}
    \otimes L^{md_{r-1}}_x.
  \]
  Now consider the composition below:
  \[
    \begin{tikzpicture}
      \matrix(m)[
        matrix of math nodes,
        row sep=-0.1cm,
        column sep=-0.15cm,
        column 1/.style={anchor=base east},
        column 2/.style={anchor=base},
        column 3/.style={anchor=base west},
        column 5/.style={anchor=base east},
        column 6/.style={anchor=base},
        column 7/.style={anchor=base west},
        column 9/.style={anchor=base}
      ]{
        H^0(X,L^m) & \otimes & H^0(X,L^{md_{r-1}})    & \hspace{1.25cm} &
        H^0(X,L^m) & \otimes & H^0(X,L^{mp^ed_{r-1}}) & \hspace{1.25cm} &
        H^0(X,L^{md_r})\\
        t_1 & \otimes & t_2 & & t_1 & \otimes & t_2^{p^e} & & t_1t_2^{p^e}\\
      };
      \path[commutative diagrams/.cd, every arrow, every label]
        (m-1-3) edge node{$1 \otimes (F^e)^*$}    (m-1-5)
        (m-1-7) edge node{$\text{mult}$}          (m-1-9)
        (m-2-3) edge[commutative diagrams/mapsto] (m-2-5)
        (m-2-7) edge[commutative diagrams/mapsto] (m-2-9);
    \end{tikzpicture}
  \]
  Note that $t_1t_2^{p^e}$ restricts to $t_{1,x}t_{2,x}^{p^e}$ in
  the stalk $L^{md_r}_x$, and that
  \begin{equation}\label{eq:liftafterfrob}
    \prod_{i=1}^n \widetilde{y}_i^{p^ea_{i}'} - t_{2,x}^{p^e} \in
    (\fm_x^{\ell+1})^{[p^{re}]} \otimes L^{mp^e d_{r-1}}_x.
  \end{equation}
  Then, after passing to the stalk $L^{md_{r}}_x$, we have
  \begin{align*}
    \prod_{i=1}^n \widetilde{y}_i^{a_i} - t_{1,x}t_{2,x}^{p^e} &= \prod_{i=1}^n \widetilde{y}_i^{a_i} -
    t_{1,x}\prod_{i=1}^n \widetilde{y}_i^{p^ea_i'} + t_{1,x}\prod_{i=1}^n \widetilde{y}_i^{p^ea_i'} -
    t_{1,x}t_{2,x}^{p^e}\\
    &= \Biggl(\prod_{i=1}^n \widetilde{y}_i^{a_{i,0}} - t_{1,x}\Biggr) \cdot \prod_{i=1}^n
    \widetilde{y}_i^{p^ea_i'} + t_{1,x} \Biggl( \prod_{i=1}^n \widetilde{y}_i^{p^ea_i'} -
    t_{2,x}^{p^e} \Biggr).
  \end{align*}
  To show that $\prod_{i=1}^n y_i^{a_i}$ is in the image of $\varphi_r$, it
  suffices to show that the right-hand side of this equation is in the image of
  $\varphi_r$ modulo $(\fm_x^{\ell+1})^{[p^{re}]}$. By \eqref{eq:liftafterfrob}, we know the second term is congruent
  to zero modulo $(\fm_x^{\ell+1})^{[p^{re}]}$, and so it remains to show the
  first term is in the image of $\varphi_r$ modulo $(\fm_x^{\ell+1})^{[p^{re}]}$.
  \par First, for each monomial $\mu$ in the $\widetilde{y}_i$ that appears in the
  difference $\prod_{i=1}^n
  \widetilde{y}_i^{a_{i,0}} - t_{1,x}$, there exists some $n$-tuple 
  $(b_i)_{1\le i \le n}$ where $\sum_{i=1}^n b_i = \ell + 1$ such that
  $\prod_{i=1}^n \widetilde{y}_i^{p^eb_i}$ divides $\mu$ by \eqref{eq:t1lift}. Thus, the
  corresponding monomial that appears in the product
  \begin{equation}\label{eq:isitinimage}
    \Biggl(\prod_{i=1}^n \widetilde{y}_i^{a_{i,0}} - t_{1,x}\Biggr) \cdot \prod_{i=1}^n
    \widetilde{y}_i^{p^ea_i'}
  \end{equation}
  is divisible by the product $\prod_{i=1}^n \widetilde{y}_i^{p^e(a_i'+b_i)}$. Since
  \[
    \sum_{i=1}^n (a_i'+b_i) = S + \ell +1 > S,
  \]
  each monomial that appears in the
  product \eqref{eq:isitinimage} is therefore in the image of $\varphi_r$ modulo
  $(\fm_x^{\ell+1})^{[p^{re}]}$ by the inductive hypothesis on $S$.
\end{proof}
This allows us to show that the supremum in \eqref{eq:lthfrobsesh} can actually
be computed as a limit supremum.
\begin{proposition}[cf.\ {\cite[Prop.\ 2.6]{MS14}}]\label{prop:ms26}
  Let $\ell \ge 0$ be fixed, and let $L$ be an ample line bundle on a projective
  variety $X$. Let $x \in X$ be a closed point.
  \begin{enumerate}[label=$(\roman*)$,ref=\ensuremath{\roman*}]
    \item\label{prop:ms26i} The line bundle $L^m$ separates $p^e$-Frobenius
      $\ell$-jets at $x$ for some positive integers $m$ and $e$.
    \item\label{prop:ms26ia} We have
      \[
        \varepsilon_F^\ell(L;x) = \sup_{m,e} \frac{p^e-1}{m/(\ell+1)},
      \]
      where the supremum is taken over all positive integers $m$ and $e$ such
      that $L^m$ separates $p^e$-Frobenius $\ell$-jets at $x$.
    \item\label{prop:ms26ii}
      Given any $\delta > 0$, there is a positive integer $e_0$ such that for
      every positive integer $e$ divisible by $e_0$, there is a positive integer
      $m$ such that $L^m$ separates $p^e$-Frobenius $\ell$-jets at $x$ and
      \begin{equation}\label{eq:limsuppropineq}
        \frac{p^e-1}{m/(\ell+1)} > \varepsilon_F^\ell(L;x) - \delta.
      \end{equation}
    \item\label{prop:ms26iii} We have
      \[
        \varepsilon_F^\ell(L;x) = \limsup_{m\to \infty}
        \frac{p^{s_F^\ell(L^m;x)}-1}{m/(\ell+1)}.
      \]
  \end{enumerate}
\end{proposition}
\begin{proof}
  For $(\ref{prop:ms26i})$, let $m \ge 1$ be such that $L^m$ is very ample, and
  let $n$ be the number of generators of $\fm_x \cdot \cO_{X,x}$. Then, $L^m$
  separates tangent directions (i.e., $1$-jets) and so $L^{m(\ell p^e +
  n(p^e-1))}$ separates $(\ell p^e + n(p^e-1))$-jets by the
  superadditivity property \eqref{eq:superadditive}. Thus, $L^{m(\ell p^e +
  n(p^e-1))}$ separates $p^e$-Frobenius $\ell$-jets at $x$ by the inclusion
  $\fm_x^{\ell p^e + n(p^e-1)+1} \subseteq (\fm_x^{\ell+1})^{[p^e]}$ in
  Lemma \ref{lem:monomials}.
  \par Assertion $(\ref{prop:ms26ia})$ follows by $(\ref{prop:ms26i})$ and the
  definition of the $\ell$th Frobenius--Seshadri constant, since
  $s_F^\ell(L^m;x)$ is defined as the maximum $e \ge 0$ such that $L^m$ separates
  $p^e$-Frobenius $\ell$-jets at $x$.
  \par For $(\ref{prop:ms26ii})$, there exist positive integers $m_0$ and $e_0$
  such that the inequality \eqref{eq:limsuppropineq} holds by
  $(\ref{prop:ms26i})$ and the definition of $\varepsilon_F^\ell(L;x)$.
  For each multiple $e = re_0$ of $e_0$, let $m
  = m_0 \frac{p^{e}-1}{p^{e_0} - 1}$. Then, by Lemma \ref{lem:tensorpower}, we
  have that $L^m$ separates $p^e$-Frobenius $\ell$-jets at $x$. The inequality
  \eqref{eq:limsuppropineq} then follows, since
  \[
    \frac{p^e-1}{m/(\ell+1)} = \frac{p^{e_0}-1}{m_0/(\ell+1)} >
    \varepsilon_F^\ell(L;x) - \delta.
  \]
  \par For $(\ref{prop:ms26iii})$, let $m_0$ and $e_0$ be as in
  $(\ref{prop:ms26ii})$ for $\delta = 1$. We
  inductively choose an increasing sequence of positive integers $(m_r)_{r \ge
  0}$ as follows: having chosen $m_r$, we choose $m_{r+1}$
  such that \eqref{eq:limsuppropineq} holds with $\delta = 1/(r+1)$, and such
  that $m_r < m_{r+1}$. Note that this increasing property
  can be ensured by the fact that the $m$'s in $(\ref{prop:ms26ii})$ increase as
  $e$ increases. Then, we have
  \[
    \varepsilon^\ell_F(L;x) = 
    \lim_{r \to \infty} \frac{p^{s_F^\ell(L^{m_r};x)}-1}{m_r/(\ell+1)},
  \]
  which implies $(\ref{prop:ms26iii})$.
\end{proof}
We now give a calculation of both the ordinary Seshadri constant and the
$\ell$th Frobenius--Seshadri constants on projective space, which
shows that the limit supremum in 
Proposition \ref{prop:ms26}$(\ref{prop:ms26iii})$ cannot be computed as a limit.
\begin{example}\label{ex:projectivespace}
  Let $X = \PP^n_k$, and let $L = \cO_X(1)$ be the line bundle associated to a
  hyperplane. For every closed point $x \in \PP^n_k$, we claim that the
  restriction map
  \[
    \rho^{\ell,e}_{\cO_X(m)}\colon H^0(X,\cO_X(m)) \longrightarrow
    H^0\bigl(X,\cO_X(m) \otimes \cO_X/(\fm_x^{\ell+1})^{[p^e]}\bigr)
  \]
  is surjective if and only if $m \ge \ell p^e + n(p^e - 1)$. This claim follows
  from Lemma \ref{lem:monomials} after choosing local affine coordinates in
  $\cO_{X,x}$, and observing that $H^0(X,\cO_X(m))$ maps onto all monomials of
  degree $\le m$ in $\cO_{X,x}$. Note that the inequality $m \ge \ell p^e +
  n(p^e - 1)$ is equivalent to
  \[
    \frac{m+n}{\ell+n} \ge p^e.
  \]
  Letting $e = 0$ gives $\varepsilon(\cO_X(m);x) = 1$.
  For Frobenius--Seshadri constants, the equality in Proposition
  \ref{prop:ms26}$(\ref{prop:ms26iii})$ implies
  \[
    \varepsilon^\ell_F(\cO_X(1);x) \le \limsup_{m \to \infty}
    \frac{\frac{m+n}{\ell+n} - 1}{m/(\ell+1)} = \frac{\ell+1}{\ell+n} \cdot
    \limsup_{m\to\infty} \frac{m-\ell}{m} = \frac{\ell+1}{\ell+n},
  \]
  and the reverse inequality holds by computing a lower bound for the limit
  supremum using the sequence $m_e = \ell p^e +
  n(p^e - 1)$. On the other hand, we see that the limit supremum is not
  a limit, since the sequence $m_e' = \ell p^e +
  n(p^e - 1) - 1$ gives the limit
  \[
    \limsup_{e \to \infty} \frac{p^{s_F^\ell(L^{m_e'};x)} - 1}{m_e'/(\ell+1)} =
    \frac{\ell+1}{\ell+n} \cdot \lim_{e \to \infty} \frac{p^{e-1} - 1}{p^e -
    \frac{n+1}{n+\ell}} = \frac{\ell+1}{(\ell+n)p}.
  \]
  Note that in this example, we have the equality $\frac{\ell+1}{\ell+n} \cdot
  \varepsilon(L;x) = \varepsilon_F^\ell(L;x)$. We will see in
  Proposition \ref{prop:ms212} that in fact, the inequality $\le$ always holds.
\end{example}
We can use Lemma \ref{lem:tensorpower} and Proposition
\ref{prop:ms26}$(\ref{prop:ms26iii})$ to prove that Frobenius--Seshadri
constants are also well-behaved with respect to taking tensor powers. While we
will not need this result in the sequel, the proof will use the following lemma,
which will also be useful in the proof of Theorem \ref{thm:sepljets}.
\begin{lemma}[{cf.\ \cite[Rem.\ 2.3]{MS14}}]\label{lem:tensorwithgg}
  Let $X$ be a projectivty variety, and let $x \in X$ be a closed point. Let $L$
  be a line bundle on $X$, and let $M$ be another line bundle that is globally
  generated at $x$. If for some integers $e \ge 1$ and $\ell \ge 0$, we have
  that $L$ separates $p^e$-Frobenius $\ell$-jets at $x$, then $L \otimes M$ also
  separates $p^e$-Frobenius $\ell$-jets at $x$.
\end{lemma}
\begin{proof}
  Since $M$ is globally generated at $x$, there is a global section $t \in
  H^0(X,M)$ that does not vanish at $x$. The commutative diagram
  \[
    \begin{tikzcd}
      H^0(X,L) \dar[swap]{t \otimes -}
      \rar[twoheadrightarrow]{\rho_{L}^{\ell,e}}
      & H^0\bigl(X,L \otimes \cO_X/(\fm_x^{\ell+1})^{[p^{e}]}\bigr)
      \dar{t_x \otimes -}[sloped,below]{\sim}\\
      H^0(X,L \otimes M) \rar{\rho_{L \otimes M}^{\ell,e}}
      & H^0\bigl(X,L \otimes M \otimes
      \cO_X/(\fm_x^{\ell+1})^{[p^{e}]}\bigr)
    \end{tikzcd}
  \]
  shows that the restriction map $\rho_{L \otimes M}^{\ell,e}$ is surjective,
  i.e., we have that $L \otimes M$ separates $p^e$-Frobenius $\ell$-jets.
\end{proof}
\begin{proposition}[cf.\ {\cite[Prop.\ 2.8]{MS14}}]
  Let $L$ be an ample line bundle on a projective variety $X$, and let $x \in
  X$ be a closed point. Then, for all integers $\ell \ge 0$ and $s > 0$, we have
  $\varepsilon_F^\ell(L^r;x) = r \cdot \varepsilon_F^\ell(L;x)$.
\end{proposition}
\begin{proof}
  First, we have
  \[
    \varepsilon_F^\ell(L^r;x) = r \cdot \sup_{m \ge 1}
    \frac{p^{s^\ell_F(L^{rm};x)-1}}{rm/(\ell+1)} \le
    r \cdot \sup_{m' \ge 1} \frac{p^{s^\ell_F(L^{m'};x)-1}}{m'/(\ell+1)}
    = r \cdot \varepsilon_F^\ell(L;x)
  \]
  by running through all tensor powers of $L$ instead of just the powers that
  are divisible by $r$. It therefore remains to show the opposite inequality.
  We will fix an integer $j > 0$ such that $L^{j}$ is globally generated.
  \par Let $\delta > 0$ be given, and let $m$ be a positive integer such that
  $L^m$ separates $p^e$-Frobenius $\ell$-jets for some integers $e,\ell \ge 0$
  such that
  \[
    \frac{p^e-1}{m/(\ell+1)} > \varepsilon_F^\ell(L;x) - \frac{\delta}{r}.
  \]
  Now let $i$ be a positive integer. Denoting $d_i = \frac{p^{ie}-1}{p^e-1}$,
  we have
  \begin{equation}\label{eq:homogensfineq}
    s_F(L^{md_i+j};x) \ge s_F(L^{md_i};x) \ge ie
  \end{equation}
  by Lemmas \ref{lem:tensorpower} and \ref{lem:tensorwithgg}. Now denoting
  \[
    a_i = \left\lceil \frac{md_i+j}{r} \right\rceil,
  \]
  we have that $ra_i \ge md_i+j$, hence
  \[
    \frac{p^{s_F^\ell(L^{ra_i};x)}-1}{a_i/(\ell+1)} \ge
    \frac{p^{s_F^\ell(L^{md_i+j};x)}-1}{a_i/(\ell+1)}
    \ge \frac{p^{ie}-1}{a_i/(\ell+1)} = \frac{p^e-1}{m/(\ell+1)} \cdot
    \frac{d_im}{\bigl\lceil (md_i + j)/r \bigr\rceil},
  \]
  where the second inequality is by \eqref{eq:homogensfineq}.
  Taking limit suprema as $i \to \infty$ and using Proposition
  \ref{prop:ms26}$(\ref{prop:ms26iii})$ gives
  \[
    \varepsilon_F^\ell(L^r;x) \ge \limsup_{i \to \infty}
    \frac{p^{s_F^\ell(L^{ra_i};x)}-1}{a_i/(\ell+1)} \ge
    r\cdot\varepsilon_F^\ell(L;x) - \delta.
  \]
  Since $\delta > 0$ was arbitrary, we have the inequality
  $\varepsilon_F^\ell(L^r;x) \ge r \cdot \varepsilon_F^\ell(L;x)$.
\end{proof}
\subsection{A comparison with the ordinary Seshadri constant}
In order to use the $\ell$th Frobenius--Seshadri constant to prove
Theorem \ref{thm:sepljetsdem}, we require a comparison with the ordinary
Seshadri constant. This will allow us to
deduce positivity properties of adjoint bundles in \S\ref{s:adjoint}.
\begin{proposition}[cf.\ {\cite[Prop.\ 2.12]{MS14}}]\label{prop:ms212}
  If $L$ is an ample line bundle on a projective variety $X$ of dimension $n$,
  then for every smooth point $x \in X$ and integer $\ell \ge 0$, we have the
  sequence of inequalities
  \begin{equation}\label{eq:keyineq}
    \frac{\ell+1}{\ell+n} \cdot \varepsilon(L;x) \le \varepsilon_F^\ell(L;x) \le
    \varepsilon(L;x).
  \end{equation}
  In particular, $\varepsilon_F^\ell(L;x) \to \varepsilon(L;x)$ as $\ell \to
  \infty$.
\end{proposition}
\begin{proof}
  Since $\fm_x\cdot\cO_{X,x}$ is generated by $n$ elements, we have the sequence
  of inclusions
  \begin{equation}\label{eq:inclusions}
    \fm_x^{\ell p^e+n(p^e-1)+1} \subseteq (\fm_x^{\ell+1})^{[p^e]} \subseteq
    \fm_x^{(\ell+1)p^e}
  \end{equation}
  by Lemma \ref{lem:monomials}. The right inclusion in \eqref{eq:inclusions}
  implies
  \[
    s(L^m;x) \ge (\ell+1)p^{s_F^\ell(L^m;x)} - 1 \ge
    (\ell+1)(p^{s_F^\ell(L^m;x)} - 1),
  \]
  and so the right inequality in \eqref{eq:keyineq} follows after dividing by
  $m$ throughout, and taking limit suprema as $m \to \infty$.
  \par For the left inequality in \eqref{eq:keyineq}, let $\delta > 0$ be
  given, and let $m_0$ be a positive integer such that
  \[
    \frac{s(L^{m_0};x)}{m_0} > \varepsilon(L;x) - \frac{\ell+n}{\ell+1} \cdot
    \delta.
  \]
  Given any non-negative integer $e$, denote
  \[
    d_e = \left\lceil \frac{\ell p^e + n(p^e-1)}{s(L^{m_0};x)} \right\rceil =
    \left\lceil \frac{(\ell+n) p^e - n}{s(L^{m_0};x)} \right\rceil.
  \]
  By the superadditivity property \eqref{eq:superadditive}, we have
  \[
    s(L^{m_0d_e};x) \ge d_e \cdot s(L^{m_0};x) \ge \ell p^e + n(p^e-1).
  \]
  By the left inclusion in \eqref{eq:inclusions}, this inequality implies
  \[
    s_F^\ell(L^{m_0d_e};x) \ge e,
  \]
  and therefore
  \[
    \varepsilon_F^\ell(L;x) \ge
    \frac{p^{s_F^\ell(L^{m_0d_e};x)}-1}{m_0d_e/(\ell+1)} \ge
    \frac{(\ell+1)(p^e-1)}{m_0\bigl\lceil\bigl((\ell+n) p^e -
    n\bigr)/s(L^{m_0};x) \bigr\rceil}.
  \]
  As $e \to \infty$, the right-hand side converges to
  \[
    \frac{\ell+1}{\ell+n} \cdot \frac{s(L^{m_0};x)}{m_0} >
    \frac{\ell+1}{\ell+n} \cdot \varepsilon(L;x) - \delta,
  \]
  hence we have the inequality
  \[
    \varepsilon_F^\ell(L;x) > \frac{\ell+1}{\ell+n} \cdot \varepsilon(L;x) -
    \delta
  \]
  for all $\delta > 0$. Since $\delta > 0$ was arbitrary, we obtain the left
  inequality in \eqref{eq:keyineq}.
\end{proof}
In light of Example \ref{ex:projectivespace} and Theorems
\ref{thm:sepljetsdem} and \ref{thm:sepljets}, it
seems more accurate
to think of the $\ell$th Frobenius--Seshadri constant as being closer to
$\frac{\ell+1}{\ell+n} \cdot \varepsilon(L;x)$ than to $\varepsilon(L;x)$,
just as for the zeroth Frobenius--Seshadri constant \cite[p.\ 869]{MS14}. We
also observe that Example \ref{ex:projectivespace} shows that the lower bound in
\eqref{eq:keyineq} is optimal.
\par We can also compare different Frobenius--Seshadri constants:
\begin{corollary}
  If $L$ is a line bundle on an $n$-dimensional projective variety $X$, then for
  every smooth point $x \in X$ and integers $\ell > m \ge 0$, we have
  \begin{equation*}
    \frac{\ell+1}{\ell+n} \cdot \varepsilon_F^m(L;x) \le
    \varepsilon_F^\ell(L;x) \le \frac{\ell+1}{m+1} \cdot \varepsilon_F^m(L;x)
  \end{equation*}
\end{corollary}
\begin{proof}
  If $L^r$ separates $p^e$-Frobenius $\ell$-jets at $x$, then it separates
  $p^e$-Frobenius $m$-jets at $x$, giving the right inequality. The
  left inequality follows by using Proposition \ref{prop:ms212} for different
  values of $\ell$.
\end{proof}
\subsection{Numerical invariance}
We now prove that Frobenius--Seshadri constants only depend on the
numerical equivalence class of a line bundle. This fact will not be used in the
sequel. Regularity in the proof below is in the sense of Castelnuovo and
Mumford; see \cite[Def.\ 1.8.4]{Laz04} for the definition.
\begin{proposition}[cf.\ {\cite[Prop.\ 2.14]{MS14}}]
  Let $X$ be a projective variety, and let $x \in X$ be a closed point.
  If $L_1$ and $L_2$ are numerically equivalent ample line bundles on
  $X$, then $\varepsilon_F^\ell(L_1;x) = \varepsilon_F^\ell(L_2;x)$ for all
  integers $\ell \ge 0$.
\end{proposition}
\begin{proof}
  We first claim that if $A$ is a globally generated ample line bundle, then
  there exists $m_0$ such that $A^m \otimes N$ is globally generated for all
  integers $m \ge m_0$ and nef line bundles $N$.
  First, by Fujita's vanishing theorem \cite[Thm.\ 5.1]{Fuj83}, there exists an
  integer $m_1$ such that for all integers $m \ge m_1$ and nef line bundles
  $N$, we have $H^i(X,A^m \otimes N) = 0$ for all $i > 0$. Thus, if
  $m \ge m_1 + \dim X$, then the line bundle $A^m \otimes N$ is
  $0$-regular with respect to $A$, hence is globally generated by \cite[Thm.\ 1.8.5$(i)$]{Laz04}. It therefore suffices to set $m_0 = m_1 + \dim X$.
  \par We now prove the proposition. By hypothesis, there exists a numerically
  trivial line bundle $P$ such that $L_2 \simeq L_1 \otimes P$. Applying the
  result of the previous paragraph where $A$ is a large enough power of $L_1$,
  we see that there exists a positive integer $j$ such that $L_1^j \otimes N$ is
  globally generated for all nef line bundles $N$, hence in particular, $L_1^j
  \otimes P^i$ is globally generated for all integers $i$. Now by
  Proposition \ref{prop:ms26}$(\ref{prop:ms26iii})$, there exists an increasing
  sequence $(m_r)_{r \ge 0}$ of positive integers such that
  \[
    \varepsilon_F^\ell(L_1;x) = \lim_{r \to \infty}
    \frac{p^{s^\ell_F(L_1^{m_r};x)}-1}{m_r/(\ell+1)}.
  \]
  For each integer $r \ge 0$, since $L_2^{m_r+j} \simeq L_1^{m_r} \otimes
  L_1^{j} \otimes P^{m_r+j}$ and $L_1^{j} \otimes P^{m_r+j}$ is globally
  generated, we see that
  \[
    s^\ell_F(L_1^{m_r};x) \le s^\ell_F(L_2^{m_r+j};x)
  \]
  by Lemma \ref{lem:tensorwithgg}. We therefore have that
  \[
    \frac{p^{s_F^\ell(L_1^{m_r};x)}-1}{m_r/(\ell+1)} \le
    \frac{p^{s_F^\ell(L_2^{m_r+j};x)}-1}{m_r/(\ell+1)} =
    \frac{p^{s_F^\ell(L_2^{m_r+j};x)}-1}{(m_r+j)/(\ell+1)} \cdot
    \frac{m_r+j}{m_r}.
  \]
  Since the limit of the left-hand side is $\varepsilon_F^\ell(L_1;x)$ by
  choice of the sequence $(m_r)_{r \ge 0}$, taking limit suprema as $r \to
  \infty$ throughout this inequality yields the inequality
  $\varepsilon_F^\ell(L_1;x) \le \varepsilon_F^\ell(L_2;x)$ by
  Proposition \ref{prop:ms26}$(\ref{prop:ms26iii})$. Finally, repeating
  the argument above after switching the roles of $L_1$ and $L_2$, we have the
  equality $\varepsilon_F^\ell(L_1;x) = \varepsilon_F^\ell(L_2;x)$.
\end{proof}
\section{Frobenius--Seshadri constants and adjoint bundles}
\label{s:adjoint}
We now turn to the proofs of Theorems \ref{thm:sepljetsdem} and
\ref{thm:sepljets}, which we restate below. Recall our standing
assumption that our ground field $k$ is algebraically closed and of
characteristic $p > 0$.
\begin{customthm}{A}\label{thm:sepljetsdem}
  Let $L$ be an ample line bundle on a smooth projective variety $X$ of
  dimension $n$. Let $x \in X$ be a closed point. If the inequality
  \[
    \varepsilon(L;x) > n + \ell
  \]
  holds, then $\omega_X \otimes L$ separates $\ell$-jets at $x$.
\end{customthm}
\begin{customthm}{C}\label{thm:sepljets}
  Let $L$ be an ample line bundle on a smooth projective variety $X$.
  Let $x \in X$ be a closed point. If the inequality
  \[
    \varepsilon_F^\ell(L;x) > \ell + 1
  \]
  holds, then $\omega_X \otimes L$ separates $\ell$-jets at $x$.
\end{customthm}
As we mentioned in \S\ref{s:intro}, Theorem \ref{thm:sepljetsdem} is an
immediate consequence of Theorem \ref{thm:sepljets}:
\begin{proof}[Proof of Theorem \texorpdfstring{\ref{thm:sepljetsdem}}{A}]
  The inequality $\varepsilon(L;x) > n+\ell$ implies the inequality
  $\varepsilon_F^\ell(L;x) > \ell+1$ by Proposition \ref{prop:ms212}, hence the
  assertion immediately follows from Theorem \ref{thm:sepljets}.
\end{proof}
Demailly proved the analogue of Theorem \ref{thm:sepljetsdem} in
characteristic zero by using the Kawamata--Viehweg vanishing theorem, but only
assuming that $L$ is big and nef \cite[Prop.\ 6.8$(a)$]{Dem92}. We do not know
if this assumption suffices in positive characteristic.
\begin{remark}\label{rem:comparisonwithms}
  It is possible to obtain a version of Theorem \ref{thm:sepljetsdem} using
  only the zeroth Frobenius--Seshadri constant by na\"ively inducing on the
  order of
  jets in the proof of \cite[Lem.\ 3.3]{MS14}. However, the hypothesis needed
  for separation of $\ell$-jets using this method is the stronger lower bound
  $\varepsilon(L;x) > n+n\ell$. The proof is also
  more technical and relies on Castelnuovo--Mumford regularity.
\end{remark}
There are two main ingredients in the proof of Theorem \ref{thm:sepljets}.
The first is the following reformulation of our results from \S\ref{s:defs}.
\begin{proposition}\label{prop:newelljets}
  Let $L$ be an ample line bundle on a projective variety $X$, and let $x \in X$
  be a closed point. If $\varepsilon_F^\ell(L;x) > \alpha$ for some real number
  $\alpha > 0$, then we can find positive integers $m$ and $e$ satisfying
  \begin{equation}\label{eq:alphaineq}
    \frac{p^e-1}{m} > \frac{\alpha}{\ell+1}
  \end{equation}
  such that $L^m$ separates $p^e$-Frobenius $\ell$-jets at $x$. Furthermore,
  we may take $m$ and $e$ so that the quantity
  \begin{equation}\label{eq:alphaarblarge}
    p^e - 1 -\frac{\alpha}{\ell+1}m
  \end{equation}
  is arbitrarily large.
\end{proposition}
\begin{proof}
  By Proposition \ref{prop:ms26}$(\ref{prop:ms26i})$ and
  the definition of $\varepsilon_F^\ell(L;x)$, we
  know there exist $m,e \ge 1$ such that the inequality \eqref{eq:alphaineq} is
  satisfied and $L^m$ separates $p^e$-Frobenius $\ell$-jets at $x$. Moreover, by
  applying Lemma \ref{lem:tensorpower}, we may make the
  replacements
  \[
    e \longmapsto re \qquad \text{and} \qquad
    m \longmapsto \frac{m(p^{re}-1)}{p^e-1}
  \]
  and not change the inequality \eqref{eq:alphaineq} or the condition on
  separation of jets. Thus, by applying these replacements for
  integers $r \ge 1$, the quantity in \eqref{eq:alphaarblarge} satisfies
  \[
    p^{re} - 1 - \frac{\alpha}{\ell+1} \cdot \frac{m(p^{re}-1)}{p^e-1}
    = (p^{re} - 1)\left( 1 - \frac{\alpha}{\ell+1} \cdot \frac{m}{p^e-1} \right)
    \longrightarrow \infty
  \]
  as $r \to \infty$ by the inequality \eqref{eq:alphaineq}. We can therefore 
  assume that the quantity \eqref{eq:alphaarblarge} is arbitrarily large.
\end{proof}
\par As in the proof of \cite[Thm.\ 3.1]{MS14}, the other main ingredient
in the proof of Theorem \ref{thm:sepljets} is the \textsl{Cartier operator}
or the \textsl{trace map}
\[
  T\colon F_*(\omega_X) \longrightarrow \omega_X,
\]
which is a morphism of $\cO_X$-modules. Here, $F\colon X \to X$ denotes the
(absolute) Frobenius morphism. See
\cite[\S1.3]{BK05} for the definition and basic properties of the map $T$.
Briefly, it can be defined as the trace map for relative
duality for the finite flat morphism $F$ as in \cite[Ch.\ III, \S6]{Har66}.
We note that $F$ is finite since $k$ is perfect, and $F$ is flat by Kunz's
theorem \cite[Lem.\ 1.1.1]{BK05} since $X$ is smooth. The trace map satisfies
the following key properties needed for our proof:
\begin{enumerate}[label=$(\alph*)$,ref=\ensuremath{\alph*}]
  \item\label{list:trace1}
    The trace map $T$ and its iterates $T^e\colon F_*^e(\omega_X)
    \to \omega_X$ are surjective \cite[Thm.\ 1.3.4]{BK05};
  \item If $\fa \subseteq \cO_X$ is a coherent ideal sheaf, then
    $T^e$ satisfies the equality
    \begin{equation}\label{eq:tracemapideals}
      T^e\bigl(F_*^e(\fa^{[p^e]}\cdot \omega_X)\bigr) = \fa
      \cdot T^e\bigl(F^e_*(\omega_X)\bigr) = \fa \cdot \omega_X.
    \end{equation}
    This follows from $(\ref{list:trace1})$ by considering the $\cO_X$-module
    structure on $F_*^e(\omega_X)$.
\end{enumerate}
\begin{remark}\label{rem:singvar}
  The surjectivity of the trace map in $(\ref{list:trace1})$ is part of the
  definition for what are called \textsl{$F$-injective} varieties \cite[Def.\
  $2.10(iv)$]{Sch14}, as long as we interpret $\omega_X$ as the cohomology sheaf
  $\mathbf{h}^{-\dim X}\omega_X^\bullet$ of the dualizing complex
  $\omega_X^\bullet$. $F$-injective varieties are related to varieties with Du
  Bois singularities in characteristic zero \cite{Sch09a}. Since the
  justification for \eqref{eq:tracemapideals} still works in this generality,
  our proof of Theorem \ref{thm:sepljets} still works for $F$-injective
  varieties.
\end{remark}
\par We are now ready to prove Theorem \ref{thm:sepljets}. The proof closely
follows that of \cite[Thm.\ $3.1(i)$]{MS14}. The idea is the following: We can
increase powers on $L$ freely so that $L^m$ separates $p^e$-Frobenius
$\ell$-jets, and then tensor by an appropriate product of the form $\omega_X
\otimes L^{p^e - m}$ that is globally generated. Then, $\omega_X \otimes
L^{p^e}$ separates $p^e$-Frobenius $\ell$-jets. The $e$th iterate $T^e$ of the
trace map $T$ allows us to take out these factors of $p^e$, and thereby
deduce that $\omega_X \otimes L$ separates $\ell$-jets.
\begin{proof}[Proof of Theorem \ref{thm:sepljets}]
  Let $\fm_x$ denote the defining ideal of $x$. By Proposition
  \ref{prop:newelljets}, we can find $m$ and $e$ such that $m < p^e - 1$ and the
  restriction map
  \[
    \rho_{L^m}^{\ell,e}\colon H^0(X,L^m) \longrightarrow H^0\bigl(X,L^m \otimes
    \cO_X/(\fm_x^{\ell+1})^{[p^e]}\bigr)
  \]
  is surjective; moreover, we may assume that $p^e - 1 - m$ is
  arbitrarily large. In particular, we may assume that $p^e - m$ is arbitrarily
  large, so that $\omega_X \otimes L^{p^e - m}$ is globally generated. By Lemma
  \ref{lem:tensorwithgg}, we then have that $\omega_X \otimes L^{p^e}$ separates
  $p^e$-Frobenius $\ell$-jets, i.e., the restriction map
  \begin{equation}\label{eq:phisur}
    \varphi\colon H^0(X,\omega_X \otimes L^{p^e}) \longrightarrow
    H^0\bigl(X,\omega_X \otimes L^{p^e} \otimes
    \cO_X/(\fm_x^{\ell+1})^{[p^{e}]}\bigr)
  \end{equation}
  is surjective.
  \par We now use the surjectivity of the $e$th iterate
  $T^e\colon F^e_*(\omega_X) \to \omega_X$ of the trace map $T$. By
  \eqref{eq:tracemapideals}, the map $T^e$ induces a surjective morphism
  \begin{align}
    F_*^e\bigl((\fm_x^{\ell+1})^{[p^e]}\cdot\omega_X\bigr)
    &\longtwoheadrightarrow \fm_x^{\ell+1} \cdot \omega_X.\nonumber
  \intertext{Tensoring this by $L$ and applying the projection formula yields a
  surjective morphism}
    F_*^e\bigl((\fm_x^{\ell+1})^{[p^e]}\cdot\omega_X \otimes L^{p^e}\bigr)
    &\longtwoheadrightarrow \fm_x^{\ell+1} \cdot \omega_X \otimes L.
    \label{eq:tracemapprojformula}
  \end{align}
  Since the Frobenius morphism $F$ is affine, the pushforward
  functor $F_*^e$ is exact,
  hence we obtain the exactness of the left column in the following commutative
  diagram:
  \[
    \begin{tikzcd}
      0\dar & 0\dar\\
      F_*^e\bigl((\fm_x^{\ell+1})^{[p^e]}\cdot \omega_X \otimes L^{p^e}\bigr)
      \rar\dar & \fm_x^{\ell+1} \cdot \omega_X \otimes L\dar\\
      F_*^e(\omega_X \otimes L^{p^e}) \rar \dar & \omega_X \otimes L\dar\\
      F_*^e\bigl(\omega_X \otimes L^{p^e} \otimes
      \cO_X/(\fm_x^{\ell+1})^{[p^e]}\bigr) \rar \dar
      & \omega_X \otimes L \otimes \cO_X/\fm_x^{\ell+1} \dar\\
      0 & 0
    \end{tikzcd}
  \]
  The top horizontal arrow is the map in \eqref{eq:tracemapprojformula}; the
  middle horizontal arrow is obtained analogously from $T^e$ by tensoring with
  $L$, and is therefore surjective. The surjectivity of the middle horizontal
  arrow also implies the bottom horizontal arrow is surjective.
  Finally, by taking global sections in the bottom square, we obtain the
  following commutative square:
  \[
    \begin{tikzcd}
      H^0(X,\omega_X \otimes L^{p^e}) \dar[twoheadrightarrow,swap]{\varphi}\rar
      & H^0(X,\omega_X \otimes L) \dar{\rho_{\omega_X\otimes L}^{\ell,0}}\\
      H^0\bigl(X,\omega_X \otimes L^{p^e} \otimes
        \cO_X/(\fm_x^{\ell+1})^{[p^{e}]}\bigr)
      \rar[twoheadrightarrow]{\psi}
      & H^0(X,\omega_X \otimes L \otimes \cO_X/\fm_x^{\ell+1})
    \end{tikzcd}
  \]
  Note that $\psi$ is surjective because the kernel of the corresponding
  morphism of sheaves is a skyscraper sheaf supported at $x$.
  We have already shown that the restriction map $\varphi$ is
  surjective in \eqref{eq:phisur}, hence $\rho_{\omega_X\otimes L}^{\ell,0}$ is
  necessarily surjective. This shows $\omega_X \otimes L$ indeed separates
  $\ell$-jets at $x$.
\end{proof}
\begin{remark}
  It is possible to define a multi-point version of $\varepsilon_F^\ell(L;x)$
  following \cite{MS14}, which would capture how $\omega_X
  \otimes L$ simultaneously separates higher-order jets at
  different points. This
  method does not improve the result of \cite[Thm.\ 3.1$(iii)$,$(iv)$]{MS14},
  which says the following:
  \begin{enumerate}[label=$(\alph*)$,ref=\ensuremath{\alph*}]
    \item If $\varepsilon_F^0(L;x) > 2$ at some closed point $x \in X$, then
      $\omega_X \otimes L$ is very big, i.e., the rational map defined by
      $\omega_X \otimes L$ is birational onto its image;
    \item If $\varepsilon_F^0(L;x) > 2$ at all closed points $x \in X$, then
      $\omega_X \otimes L$ is very ample.
  \end{enumerate}
\end{remark}

\section{Characterizations of projective space}\label{s:charpn}
We now give an application of our result on separation of jets. As far as we
know, this is the first application of the methods of \cite{MS14}.
\par Recall that a Fano variety is a projective variety whose
anti-canonical bundle $\omega_X^{-1}$
is ample. Using Seshadri constants, Bauer and Szemberg showed the following
characterization of projective space amongst smooth Fano varieties:
\begin{citedthm}[{\cite[Thm.\ 2]{BS09}}]\label{thm:bs09}
  Let $X$ be a smooth Fano variety of dimension $n$ defined over an
  algebraically closed field of characteristic zero. If there exists a closed
  point $x \in X$ with
  \[
    \varepsilon(\omega_X^{-1};x) \ge n + 1,
  \]
  then $X$ is isomorphic to the $n$-dimensional projective space $\PP^n$.
\end{citedthm}
\par Our goal in this section is to prove the following positive characteristic
version of this result.
\begin{customthm}{B}\label{thm:charpn}
  Let $X$ be a smooth Fano variety of dimension $n$ defined over an
  algebraically closed field of positive characteristic. If there exists a
  closed point $x \in X$ with
  \[
    \varepsilon(\omega_X^{-1};x) \ge n+1,
  \]
  then $X$ is isomorphic to the $n$-dimensional projective space $\PP^n$.
\end{customthm}
\subsection{Comparison with other results and a weaker
statement}\label{s:charpncomparison}
Before moving on to the proof of Theorem \ref{thm:charpn}, we compare our
result to other characterizations of projective space. As a consequence of
these other characterizations, we also prove a weaker
statement (Proposition \ref{prop:charpnallpts}) to
illustrate why we might expect lower bounds on Seshadri constants to
give characterizations of projective space.
\par Let $K_X$ denote the canonical divisor on $X$.
Theorem \ref{thm:charpn} can be restated as follows:
\begin{theorem}\label{thm:altcharpn}
  Let $X$ be a smooth Fano variety of dimension $n$ defined over an
  algebraically closed field of positive characteristic. If there exists a
  closed point $x \in X$ with
  \[
    (-K_X \cdot C) \ge (\mult_x C) \cdot (n+1)
  \]
  for all reduced and irreducible curves $C \subseteq X$ passing through $x$,
  then $X$ is isomorphic to the $n$-dimensional projective space $\PP^n$.
\end{theorem}
\begin{proof}
  This follows from Theorem \ref{thm:charpn} by using \cite[Props.\ 5.1.5,
  5.1.17]{Laz04}, which say
  \begin{equation}\label{eq:seshinfcurves}
    \varepsilon(\omega_X^{-1};x) = \inf_{x \in C \subseteq X} \biggl\{
    \frac{(-K_X \cdot C)}{\mult_{x} C}\biggr\},
  \end{equation}
  where the infimum is taken over all reduced and irreducible curves $C
  \subseteq X$ passing through $x$.
\end{proof}
This formulation is reminiscent of the following conjecture due to Mori and
Mukai, which we mentioned in \S\ref{s:intro}:
\begin{citedconj}[{\cite[Conj.\ V.1.7]{Kol96}}]\label{conj:morimukai}
  Let $X$ be a smooth Fano variety of dimension $n$ defined over an
  algebraically closed field. If the inequality
  \[
    (-K_X \cdot C) \ge n+1
  \]
  holds for every rational curve $C \subseteq X$, then $X$ is isomorphic to
  the $n$-dimen\-sional projective space $\PP^n$.
\end{citedconj}
By using results of Kebekus \cite{Keb02} on families of singular rational
curves, Cho, Miyaoka, and Shepherd-Barron proved this conjecture in
characteristic zero. More precisely, they showed the following stronger
statement:
\begin{citedthm}[{\cite[Cor.\ 0.4]{CMSB02}}]\label{thm:cmsb}
  Let $X$ be a smooth projective variety of dimension $n$ defined over an
  algebraically closed field of characteristic zero. If $X$ is uniruled, and
  the inequality
  \[
    (-K_X \cdot C) \ge n+1
  \]
  holds for every rational curve $C \subseteq X$ passing through a general point
  $x_0$, then $X$ is isomorphic to the $n$-dimensional projective space $\PP^n$.
\end{citedthm}
In arbitrary characteristic, as far as we know the only result in this direction
is the following:
\begin{citedthm}[{\cite[Cor.\ 3]{KK00}}]\label{thm:kkcharpn}
  Let $X$ be a smooth projective variety of dimension $n$ defined over an
  algebraically closed field of arbitrary characteristic. Suppose $K_X$ is not
  nef. If
  \begin{enumerate}[label=$(\alph*)$,ref=\ensuremath{\alph*}]
    \item\label{list:kkcond1} $(-K_X \cdot C) \ge n + 1$ for every
      rational curve $C \subseteq X$; and
    \item\label{list:kkcond2} $(-K_X)^n \ge (n+1)^n$,
  \end{enumerate}
  then $X$ is isomorphic to the $n$-dimensional projective space $\PP^n$.
\end{citedthm}
Given the similarity between the Mori--Mukai conjecture \ref{conj:morimukai} and
Theorem \ref{thm:altcharpn}, we asked the following question in \S\ref{s:intro}:
\begin{question*}
  Let $X$ be a smooth Fano variety of dimension $n$ defined over an
  algebraically closed field of arbitrary characteristic. If the inequality
  \[
    (-K_X \cdot C) \ge n+1
  \]
  holds for every rational curve $C \subseteq X$, then does there exist a closed
  point $x \in X$ with
  \[
    (-K_X \cdot C) \ge (\mult_x C) \cdot (n+1)
  \]
  for all reduced and irreducible curves $C \subseteq X$ passing through $x$?
\end{question*}
As mentioned in \S\ref{s:intro}, the answer to this question is ``yes'' in
characteristic zero by using Theorem \ref{thm:cmsb}, since Theorem\
\ref{thm:cmsb} implies $X \simeq \PP^n$, and therefore
$\varepsilon(\omega_X^{-1};x) \ge n+1$ for all closed points $x \in X$ by
Example \ref{ex:projectivespace}.
If one could answer this question affirmatively independently of Theorem
\ref{thm:cmsb}, then Theorem \ref{thm:bs09} would give an alternative proof of
the Mori--Mukai conjecture \ref{conj:morimukai} in characteristic zero, and
Theorem \ref{thm:altcharpn} would resolve their conjecture in positive
characteristic. 
\par Returning to Seshadri constants, we can show the following statement as a
consequence of the characterizations of projective space given above. The
statement in characteristic zero gives a different proof of \cite[Thm.\ 2]{BS09}.
\begin{proposition}\label{prop:charpnallpts}
  Let $X$ be a smooth Fano variety of dimension $n$ defined over an algebraically
  closed field $k$. Consider the inequality
  \begin{equation}\label{eq:weakervers}
    \varepsilon(\omega_X^{-1};x) \ge n+1
  \end{equation}
  for each closed point $x \in X$. Suppose one of the following is
  satisfied:
  \begin{enumerate}[label=$(\roman*)$,ref=\ensuremath{\roman*}]
    \item\label{prop:charpnallptschar0}
      We have $\Char k = 0$ and the inequality
      \eqref{eq:weakervers} holds for a single closed point $x \in X$;
      or
    \item\label{prop:charpnallptscharp}
      We have $\Char k = p > 0$ and the inequality
      \eqref{eq:weakervers} holds for \emph{all} closed points $x \in X$.
  \end{enumerate}
  Then, $X$ is isomorphic to the $n$-dimensional projective space $\PP^n$ over
  $k$.
\end{proposition}
\begin{proof}
  For $(\ref{prop:charpnallptschar0})$, we use Theorem \ref{thm:cmsb}. Since Fano
  varieties are uniruled \cite[Cor.\ IV.1.15]{Kol96}, it suffices to verify the
  condition $(-K_X \cdot C) \ge n+1$.
  First, note that $\varepsilon(\omega_X^{-1};x) > n$ at the given point $x \in
  X$, and since the locus $\bigl\{x \in X \mid
  \varepsilon(\omega_X^{-1};x) > n\bigr\}$
  is open \cite[Rem.\ 2.15]{MS14}, we have $\varepsilon(\omega_X^{-1};x_0) >
  n$ at a general point $x_0 \in X$. By the alternative characterization of
  Seshadri constants in terms of curves in \eqref{eq:seshinfcurves}, we have the
  chain of inequalities
  \[
    n < \varepsilon(\omega_X^{-1};x_0) \le \frac{(-K_X \cdot C)}{\mult_{x_0} C}
    \le (-K_X \cdot C)
  \]
  for any rational curve $C$ containing $x_0$. Since
  $(-K_X \cdot C)$ is an integer, we have $(-K_X \cdot C) \ge n+1$.
  \par For $(\ref{prop:charpnallptscharp})$, we use Theorem \ref{thm:kkcharpn}.
  The verification of condition $(\ref{list:kkcond1})$
  proceeds as in $(\ref{prop:charpnallptschar0})$ by applying
  \eqref{eq:seshinfcurves} to a closed point
  $x \in C$ contained in a given rational curve $C \subseteq X$.
  For condition $(\ref{list:kkcond2})$, we use the inequality
  \[
    \varepsilon(\omega_X^{-1};x) \le \sqrt[n]{(-K_X)^n},
  \]
  which is \cite[eq.\ 5.2]{Laz04}. The inequality $\varepsilon(\omega_X^{-1};x)
  \ge n+1$ then implies condition $(\ref{list:kkcond2})$.
\end{proof}
\subsection{Proof of Theorem \ref{thm:charpn}}
We now turn to the proof of Theorem \ref{thm:charpn}. The main technical
tool is the notion of bundles of principal parts, which are also known as jet
bundles in the literature. See \cite[\S4]{LT95} for a detailed discussion.
\begin{definition}
  Let $X$ be a variety defined over an algebraically closed field $k$ of
  arbitrary characteristic. Denote by $p$ and $q$ the projections
  \[
    \begin{tikzcd}[column sep=0]
      & X \times X\arrow{dl}[swap]{p}\arrow{dr}{q}\\
      X & & X
    \end{tikzcd}
  \]
  Let $\sI \subset \cO_{X \times X}$ be the ideal defining the diagonal, and
  let $L$ be a line bundle on $X$. For each integer $\ell \ge 0$, the
  \textsl{$\ell$th bundle of principal parts} associated to $L$ is the sheaf
  \[
    \sP^\ell(L) \coloneqq p_*(q^*L \otimes \cO_{X \times X}/\sI^{\ell+1}).
  \]
  Note that $\sP^0(L) \simeq L$, since the diagonal in $X \times X$ is isomorphic
  to $X$.
\end{definition}
We will use the following facts about these sheaves from \cite[\S4]{LT95},
assuming $X$ is smooth:
\begin{enumerate}[label=$(\alph*)$,ref=\ensuremath{\alph*}]
  \item There exists a short exact sequence
    \cite[n\textsuperscript{o}\ 4.2]{LT95}
    \begin{equation}\label{eq:sesbundlepp}
      0 \longrightarrow \Sym^\ell(\Omega_X) \otimes L \longrightarrow \sP^\ell(L)
      \longrightarrow \sP^{\ell-1}(L) \longrightarrow 0,
    \end{equation}
    where $\Omega_X$ denotes the cotangent bundle on $X$. By using
    induction and this short exact sequence, it follows that the sheaf
    $\sP^\ell(L)$ is a vector bundle for all integers $\ell \ge 0$.
  \item\label{property:bundlepp2}
    There exists an identification $\sP^\ell(L) \simeq q_*(q^*L \otimes
    \cO_{X\times X}/\sI^{\ell+1})$, and by applying adjunction to the
    map $q^*L \to q^*L \otimes \cO_{X\times X}/\sI^{\ell+1}$, there is a
    morphism
    \[
      d^\ell\colon L \longrightarrow \sP^\ell(L)
    \]
    of sheaves \cite[n\textsuperscript{o}\ 4.1]{LT95}, such that the
    diagram
    \[
      \begin{tikzcd}
         H^0(X,L) \rar{H^0(d^\ell)}\dar[swap]{\rho^{\ell,0}_L} &
         H^0\bigl(X,\sP^\ell(L)\bigr) \dar{\rho^{0,0}_{\sP^\ell(L)}}\\
         H^0(X,L \otimes \cO_X/\fm_x^{\ell+1}) &
         H^0\bigl(X,\sP^\ell(L) \otimes \cO_X/\fm_x\bigr) \lar[swap]{\sim}
      \end{tikzcd}
    \]
    commutes for all closed points $x \in X$ \cite[Lem.\ 4.5(1)]{LT95}.
    Thus, if $L$ separates $\ell$-jets at $x$, then $\sP^\ell(L)$ is globally
    generated at $x$.
\end{enumerate}
\par We will also use the following description of the determinant of the
$\ell$th bundle of principal parts. This
description is stated in \cite[p.\ 1660]{DRS01}.
\begin{lemma}\label{lem:detofpp}
  Let $X$ be a smooth variety of dimension $n$, and let $L$ be a line bundle.
  Then, for each $\ell \ge 0$, we have an isomorphism
  \[
    \det(\sP^{\ell}(L)) \simeq \bigl( \omega_X^\ell \otimes L^{n+1}
    \bigr)^{\frac{1}{n+1}\binom{n+\ell}{n}}.
  \]
\end{lemma}
\begin{proof}
  We proceed by induction on $\ell \ge 0$. If $\ell = 0$, then $\sP^0(L) \simeq
  L$, so we are done.
  \par Now suppose $\ell > 0$. Since $X$ is smooth, the cotangent bundle
  $\Omega_X$ has rank $n$, and we have isomorphisms
  \[
    \det\bigl(\Sym^\ell(\Omega_X) \otimes L\bigr)
    \simeq \det\bigl(\Sym^\ell(\Omega_X)\bigr) \otimes
    L^{\binom{n+\ell-1}{n-1}}_{\vphantom{X}}
    \simeq \omega_X^{\binom{n+\ell-1}{n}} \otimes
    L^{\binom{n+\ell-1}{n-1}}_{\vphantom{X}}.
  \]
  By induction and taking top exterior powers in the short exact sequence
  \eqref{eq:sesbundlepp}, we obtain
  \begin{align*}
    \det(\sP^{\ell}(L)) &\simeq \omega_X^{\binom{n+\ell-1}{n}} \otimes
    L^{\binom{n+\ell-1}{n-1}}_{\vphantom{X}} \otimes \det(\sP^{\ell-1}(L))\\
    &\simeq\omega_X^{\binom{n+\ell-1}{n}} \otimes
    L^{\binom{n+\ell-1}{n-1}}_{\vphantom{X}} \otimes \bigl( \omega_X^{\ell-1}
    \otimes L^{n+1} \bigr)^{\frac{1}{n+1}\binom{n+\ell-1}{n}}\\
    &\simeq \bigl( \omega_X^\ell \otimes L^{n+1}
    \bigr)^{\frac{1}{n+1}\binom{n+\ell}{n}}.
  \end{align*}
  Note that the last isomorphism holds because of the identities
  \begin{gather*}
    \binom{n+\ell-1}{n} + \frac{\ell-1}{n+1}\binom{n+\ell-1}{n} =
    \frac{n+\ell}{n+1} \binom{n+\ell-1}{n} =
    \frac{\ell}{n+1}\binom{n+\ell}{n}\\[0.5em]
    \binom{n+\ell-1}{n-1} + \binom{n+\ell-1}{n} = \binom{n+\ell}{n}
  \end{gather*}
  involving binomial coefficients.
\end{proof}
We now return to the setting where our ground field $k$
is an algebraically closed field of characteristic $p > 0$. We begin with
the following key chain of inequalities. Note that our statement is weaker than
\cite[Prop.\ 1.1]{BS09}, but it still suffices for our purposes.
\begin{lemma}\label{lem:seshineq}
  Let $X$ be a smooth Fano variety of dimension $n$, and let $x \in X$ be a
  closed point. Denote $\varepsilon = \varepsilon(\omega_X^{-1};x)$.
  For every integer $m \ge 1$, we have the chain of inequalities
  \begin{equation}\label{eq:numineq}
    (m+1)\varepsilon - (n + 1) \le s(\omega_X^{-m};x) \le m\varepsilon.
  \end{equation}
  In particular, $\varepsilon(\omega_X^{-1};x) \le n+1$.
\end{lemma}
\begin{proof}
  We have the inequality
  \[
    \frac{s(\omega_X^{-m};x)}{m} \le \varepsilon
  \]
  by the definition of the ordinary Seshadri constant in
  \eqref{eq:defordseshadri}. We can then multiply by $m$ throughout to obtain
  the right inequality in \eqref{eq:numineq}.
  \par For the left inequality in \eqref{eq:numineq}, we know that
  if $\omega_X^{-m}$ does not separate $\ell$-jets, then
  \begin{equation}\label{eq:ineqchain}
    \varepsilon(\omega_X^{-(m+1)};x) = (m+1) \cdot \varepsilon(\omega_X^{-1};x)
    \le n + \ell
  \end{equation}
  by the contrapositive of Theorem \ref{thm:sepljetsdem} applied to $L =
  \omega_X^{-(m+1)}$. Note that the equality in \eqref{eq:ineqchain} holds by
  \cite[Ex.\ 5.1.4]{Laz04}. By the definition of $s(\omega_X^{-m};x)$, the
  inequality in \eqref{eq:ineqchain} holds for $\ell =
  s(\omega_X^{-m};x) +1$, hence the left inequality in \eqref{eq:numineq}
  follows. The last assertion follows by rearranging 
  \eqref{eq:numineq} for $m = 1$.
\end{proof}
We now prove Theorem \ref{thm:charpn}. Our proof follows that of
\cite[Thm.\ 1.7]{BS09}, although we must be more careful with tensor operations
in positive characteristic.
\begin{proof}[Proof of Theorem \ref{thm:charpn}]
  We first claim that $\sP^{n+1}(\omega_X^{-1})$ is a trivial bundle.
  By Lemma \ref{lem:seshineq} for $m=1$, we know that at the given point
  $x \in X$, we have the equality $\varepsilon(\omega_X^{-1};x) = n+1$, and
  moreover
  \[
    n+1 = 2 \cdot \varepsilon(\omega_X^{-1};x) - (n+1) \le s(\omega_X^{-1};x)
    \le \varepsilon(\omega_X^{-1};x) = n+1,
  \]
  hence equality holds throughout.
  By property $(\ref{property:bundlepp2})$ of bundles of principal parts, we
  therefore have that $\sP^{n+1}(\omega_X^{-1})$ is globally
  generated at $x$. On the other hand, by Lemma \ref{lem:detofpp} applied to $L =
  \omega_X^{-1}$, we have an isomorphism $\det(\sP^{n+1}(\omega_X^{-1})) \simeq
  \cO_X$. Now to show that $\sP^{n+1}(\omega_X^{-1})$ is a trivial bundle,
  consider the following diagram:
  \[
    \begin{tikzcd}
      \det\bigl(\sP^{n+1}(\omega_X^{-1})\bigr) \rar{\sim}\dar[twoheadrightarrow]
      & \cO_X\dar[twoheadrightarrow]\\
      \det\bigl(\sP^{n+1}(\omega_X^{-1}) \otimes \cO_X/\fm_x\bigr) \rar{\sim}
      & \cO_X/\fm_x
    \end{tikzcd}
  \]
  Suppose the isomorphism in the top row is given by a non-vanishing global
  section
  \[
    s \in H^0\bigl(X,\det\bigl(\sP^{n+1}(\omega_X^{-1})\bigr)\bigr).
  \]
  Let $s_{1,x} \wedge s_{2,x} \wedge \cdots \wedge s_{r,x}$ be the image of $s$
  in $\det\bigl(\sP^{n+1}(\omega_X^{-1}) \otimes \cO_X/\fm_x\bigr)$, which gives
  the isomorphism in the bottom row. Then, since $\sP^{n+1}(\omega_X^{-1})$ is
  globally generated at $x$, each $s_{i,x}$ can be lifted to a global section
  $\widetilde{s}_i \in H^0\bigl(X,\sP^{n+1}(\omega_X^{-1})\bigr)$. Because the
  exterior product $\widetilde{s}_1 \wedge \widetilde{s}_2 \wedge \cdots \wedge
  \widetilde{s}_r$ does not vanish at $x$, this exterior product does not vanish
  anywhere, since $H^0(X,\cO_X) = k$.  Thus, the global sections
  $\widetilde{s}_i$ give a frame for $\sP^{n+1}(\omega_X^{-1})$, and therefore
  $\sP^{n+1}(\omega_X^{-1})$ is a trivial bundle.
  \par To show $X \simeq \PP^n_k$, we use Mori's characterization of projective
  space \cite[Thm.\ V.3.2]{Kol96}. It suffices to show that for
  every nonconstant morphism $f\colon \PP^1_k \to X$,
  the pull back $f^*T_X$ is a sum of line bundles of positive degree. Write
  \[
    f^*(T_X) \simeq \bigoplus_{i=1}^n \cO(a_i) \qquad \text{and} \qquad
    f^*(\omega_{X}^{-1}) \simeq \cO(b),
  \]
  where $b$ is positive since $\omega_X^{-1}$ is ample.
  We want to show that each $a_i$ is positive. Note that
  \[
    f^*(\Omega_X) \simeq f^*(T_X)^\vee \simeq \bigoplus_{i=1}^n \cO(-a_i).
  \]
  Dualizing the short exact sequence \eqref{eq:sesbundlepp}, we have the short
  exact sequence
  \[
    0 \longrightarrow \sP^n(\omega_X^{-1})^\vee \longrightarrow
    \sP^{n+1}(\omega_X^{-1})^\vee \longrightarrow (\Sym^{n+1}\Omega_X)^\vee
    \otimes \omega_X \longrightarrow 0.
  \]
  The quotient on the right is globally generated because it is a quotient of
  the trivial bundle $\sP^{n+1}(\omega_X^{-1})^\vee$. We have
  isomorphisms
  \begin{align*}
    f^*\bigl((\Sym^{n+1}\Omega_X)^\vee \otimes \omega_X\bigr) &\simeq
    \bigl(\Sym^{n+1} f^*(\Omega_X) \bigr)^\vee \otimes f^*(\omega_X)\\
    &\simeq
    \biggl(\Sym^{n+1} \bigoplus_{i=1}^n \cO(-a_i) \biggr)^\vee \otimes \cO(-b),
  \end{align*}
  and this bundle is globally generated since it is the pullback of a globally
  generated bundle. By expanding out the
  symmetric power on the right-hand side, we have a surjection
  \[
    f^*\bigl((\Sym^{n+1}\Omega_X)^\vee \otimes \omega_X\bigr)
    \longtwoheadrightarrow
    \bigoplus_{i=1}^n \cO\bigl((n+1)a_i-b\bigr),
  \]
  hence the direct sum on the right-hand side is also globally generated.
  Finally, this implies
  \[
    (n+1)a_i - b \ge 0,
  \]
  and therefore since $b > 0$, we have that $a_i > 0$ as required.
\end{proof}
\begin{remark}\label{rem:lz16}
  Liu and Zhuang's characteristic zero statement in \cite[Thm.\ 2]{LZ16} is
  stronger than Theorem \ref{thm:charpn}: it only assumes that $X$ is
  $\QQ$-Fano, and in particular that $X$ is not necessarily smooth. While
  Theorem \ref{thm:sepljets} holds for a large class of singular varieties
  (see Remark \ref{rem:singvar}), the rest of our approach does not generalize
  to the non-smooth setting, since
  Mori's characterization of projective space uses bend and
  break techniques. On the other hand, Liu and Zhuang's methods do not seem to
  work in positive characteristic without very strong assumptions on dimension
  and $F$-singularities since, in particular, they use the Kawamata--Shokurov
  basepoint-freeness theorem and the Kawamata--Viehweg vanishing theorem.
\end{remark}

\section*{Acknowledgments}
I am grateful to my advisor Mircea Musta\c{t}\u{a} for his
support and for numerous fruitful conversations, and to Rankeya Datta, Emanuel
Reinecke, and Matthew Stevenson for helpful comments on drafts
of this paper. I would also like to thank J\'anos Koll\'ar and Ziquan Zhuang for
insights on their results on characterizations of projective space. Finally, I
am indebted to the anonymous referee for useful suggestions that improved the
quality of this paper.

\end{document}